\documentclass[a4paper,11pt]{article}
\usepackage{amsmath}
\usepackage{amsthm}
\usepackage{amssymb}
\usepackage[latin1]{inputenc}
\usepackage[english]{babel}

\theoremstyle{plain}
\newtheorem{thm}{Theorem}[section]

\newtheorem{lem}[thm]{Lemma}
\newtheorem{cor}[thm]{Corollary}
\theoremstyle{definition}
\newtheorem{defn}[thm]{Definition}
\theoremstyle{remark}
\newtheorem{ex}[thm]{Example}
\newtheorem{rem}[thm]{Remark}

\title{CM--stability of blow--ups and canonical metrics \\ {\normalsize \it dedicated to the memory of prof. G. Bassanelli}}
\author{Alberto Della Vedova \footnote{\noindent {\bf Address.} Dipartimento di Matematica, Universit\`a degli Studi di Parma, Viale G. P. Usberti, 53/A - 43100 Parma (Italy). {\bf e-mail}: alberto.dellavedova@unipr.it}}
\date{}

\begin{document}
\maketitle

\begin{abstract}
\noindent An asymptotic formula for the Tian--Paul CM--line of a flat family blown--up at a flat closed sub-scheme is given. As an application we prove that the blow--up of a polarized manifold along a (relatively) Chow--unstable submanifold admits no (extremal) constant scalar curvature K\"ahler metrics in classes making the exceptional divisors sufficiently small. Moreover a geometric characterization of relatively Chow--unstable configuration of points in the projective space is given. From this we get new examples of classes admitting no extremal K\"ahler metric also in the case of the projective plane blown--up at a finite set of points.

\medskip 

\end{abstract}

\section{Introduction and statement of results}

The problem of finding a canonical metric on a fixed K\"ahler class of a compact manifold has a rich and long history mainly due to Calabi, Aubin, Yau, Tian and Donaldson. First general results on K\"ahler-Einstein metrics are due to Aubin \cite{Aub76} in the negative first Chern class case and Yau \cite{Yau78} in the non--positive case \cite{Yau78} with the celebrated proof of the Calabi conjecture. Since Matsushima theorem \cite{Mat57} and the Futaki invariant \cite{Fut83}, the positive (Fano) case is known to be obstructed and in general the problem is still open. An insight of Yau relates the existence of such a metric on a Fano manifold to some kind of algebraic--geometric stability. In \cite{Tia97} Tian defined K--stability for Fano manifolds, which gives a subtle obstruction to the existence of K\"ahler-Einstein metrics, and made concrete Yau's suggestion introducing CM--polarization for families of Fano manifolds and formulating a precise conjecture relating the GIT stability relative to CM--polarization to the existence of K\"ahler-Einstein metrics. Later Donaldson introduced his K--stability for general polarized manifolds (with non necessarily anti--canonical polarization) and conjectured it is equivalent to the existence of a cscK (constant scalar curvature K\"ahler) metric in the polarization class \cite{Don02}. Finally one step further has been made by Sz\'ekelyhidi \cite{Sze06} introducing relative K-stability for eK (extremal K\"ahler) metrics. Calabi defined such metrics, generalizing cscK metrics, with the aim to give a canonical representative in each K\"ahler class of a compact complex manifold  \cite{Cal82,Cal85}, but after Levine \cite{Lev85} and Burns--De Bartolomeis \cite{BurDeB88} we know that there exist K\"ahler manifolds admitting no extremal metrics in any K\"ahler class. More recently, has been shown examples \cite{Ton98,ApoCalGauTon05} of K\"ahler manifold having extremal metrics only in some classes. The main application of results of this paper is the construction of new examples of manifolds which admit no eK metrics in some classes. 

The equivalence between the existence of a cscK (or more generally eK) metric in the polarization class and the algebraic--geometric stability of a polarized manifold sometimes goes under the name of {\it Yau--Tian--Donaldson conjecture}. Some important steps toward the proof of such conjecture have been done. In particular the stability of a polarized manifold admitting a canonical metric has been proved \cite{Tia97, Don05, Sto08bis, Sze06, Sze08}. Moreover one expects that the stability condition involved in the Yau--Tian--Donaldson conjecture is given by the refined CM--polarization of the Hilbert scheme introduced by Paul--Tian \cite{PauTia06}, which is in fact deeply related to Donaldson's K-stability \cite{Don02}. Actually Fine and Ross showed that the line bundle introduced by Paul and Tian in general is not ample \cite{FinRos06}, thus we will refer to it (or to a rational multiple) as the {\it CM--line}.

The main result of this paper is quite technical, but it has a number of applications to the problem of finding canonical metrics. Roughly speaking it is the expression of the CM--line of a blown-up family in terms of the CM--line of the original family and a kind of Chow--stability of the center of blow--up. To be more precise, let $\pi : X \to B$ be a flat family of relative dimension $n$ endowed with a polarization $L$ ({\it i.e.} a relatively ample line bundle on $X$). Following Paul--Tian construction \cite{PauTia06}, these data define the CM--line  $$ \lambda_{\rm CM}(X/B,L) \in \mathbb {\rm Pic}_{\mathbb Q}(B).$$ Moreover, to any closed sub--scheme $Y \subset X$ which is flat over $B$ (via $\pi$) and of relative dimension $d$ we associate the CW--line (see section \ref{sec::CM&CW}) $$ \lambda_{\rm CW}(Y,X/B,L) \in \mathbb {\rm Pic}_{\mathbb Q}(B).$$ Now consider the blow--up $\beta: \tilde X \to X$ of $X$ with center $Y$ endowed with the polarization $L_r = \beta^*L^r \otimes \mathcal O_{\tilde X}(1)$ for $r$ sufficiently large. The family $\tilde \pi = \pi \circ \beta : \tilde X \to B$ is flat, thus we can consider the CM--line $\lambda_{\rm CM}\left( \tilde X/B, L_r \right) \in {\rm Pic}_{\mathbb Q}(B)$. Our main result is the following 

\begin{thm}\label{thm::main_thm_intro} For $r$ sufficiently large we have
\begin{equation}\label{eq::main_thm_intro} \lambda_{\rm CM}\left( \tilde X/B, L_r\right) = \lambda_{\rm CM} \left( X/B, L\right) \otimes \lambda_{\rm CW} \left( Y,X/B,L \right)^\frac{1}{r^{n-d-1}} \otimes O\left( \frac{1}{r^{n-d}} \right),\end{equation} where $O\left( \frac{1}{r^{n-d}} \right) = \bigotimes_{i\geq n-d} \epsilon_i^\frac{1}{r^i}$ for some fixed $\mathbb Q$-line bundles $\epsilon_i$ on $B$.
\end{thm}

As we anticipated above, formula \eqref{eq::main_thm_intro} has many applications to the problem of existence of canonical metrics on polarized manifolds. To state them we apply theorem \ref{thm::main_thm_intro} to a test configuration for a polarized manifold $(M,A)$ (as defined in \cite{Don02}). In other words we make the following assumptions:
\begin{itemize} 
\item $B = \mathbb C$ endowed with the natural $\mathbb C^\times$--action,
\item are given a $\mathbb C^\times$--action on $X$ making $\pi$ an equivariant map and a linearization on $L$, 
\item for any fiber $X_t=\pi^{-1}(t)$ over $t \neq 0$ we have $(X_t,L|_{X_t}) \simeq (M,A)$ as polarized manifolds. 
\end{itemize}
From these data we get a polarized scheme endowed with a $\mathbb C^\times$-action, namely the central fiber $X_0=\pi^{-1}(0)$ of the family, polarized with $L|_{X_0}$; with a little abuse of notation we denote by $F(X,L)$ the generalized Futaki invariant of $(X_0,L|_{X_0})$ as defined by Donaldson \cite{Don02} for a general scheme. In addition suppose that 
\begin{itemize}
\item the sub--scheme $Y \subset X$ is invariant,
\item each fiber $Y_t=\pi|_Y^{-1}(t)$ is smooth (although possibly reducible or non--reduced) for $t \neq 0$.
\end {itemize}
From these data both CM and CW--lines inherit a linearization and equation \eqref{eq::main_thm_intro} holds in the sense of linearized bundles. By Paul--Tian \cite[theorem 1]{PauTia06} we know that the weight of the $\mathbb C^\times$--action on the CM--line $\lambda_{\rm CM}\left(X/\mathbb C,L\right)$ over $0 \in \mathbb C$ is nothing but the generalized Futaki invariant $F(X,L)$. Moreover let $w_{\rm CW}\left( Y,X,L \right)$ be the weight of the induced $\mathbb C^\times$--action on the CW--line $\lambda_{\rm CW}\left(Y,X/\mathbb C,L\right)$ on the fiber over $t=0$. 

In remark \ref{rem::CW_weight} we will show that the weight $w_{\rm CW}$ is in fact a generalization of the Chow weight of a projective variety introduced (although not explicitly defined) by Mumford in \cite{Mum77}. We will call $w_{\rm CW}(Y,X,L)$ the {\it Chow weight} of $Y \subset (X,L)$ (or $Y_t \subset (X_t,L|_{X_t})$). Actually $w_{\rm CW}\left(Y,X,L\right)$ depends only on the fibers of $(X,L)$ and $Y$ over $t=0$, but the reader should not confuse $w_{\rm CM}(Y,X,L)$ with the standard Chow weight of the polarized scheme $\left( Y_0,L|_{Y_0}\right)$, being $Y_0 = \pi|_Y^{-1}(0)$.

Following general construction described above, after blowing up the sub-scheme $Y$ we get the flat family $\tilde \pi : \tilde X \to \mathbb C$ polarized with $L_r$. A few comments are in order before going ahead. First, since $Y$ is invariant, the $\mathbb C^\times$--action on $X$ lifts to $\tilde X$ and the map $\tilde \pi$ is equivariant. Second, since $Y$ is transversal to each fiber $X_t$ for $t \neq 0$, the fiber $\tilde X_t = \tilde \pi^{-1}(t)$ is the blow--up of $X_t$ along $Y_t$ for $t \neq 0$. Thus $(\tilde X/\mathbb C, L_r)$ is a test-configuration for the polarized manifold $(\tilde M, \beta^*A^r \otimes \mathcal O(-E))$, where $\beta: \tilde M \to M$ is the blow--up of $M$ along $N=Y_1$ with exceptional divisor $E$. In this situation by theorem \ref{thm::main_thm_intro} we can prove the following 

\begin{cor}\label{cor::fut_exp_intro}
For $r$ sufficiently large we have $$ F\left(\tilde X, L_r\right) = F\left(X,L\right) + \frac{w_{\rm CW} \left(Y,X,L\right)}{r^{n-d-1}} + O\left(\frac{1}{r^{n-d}}\right). $$ 
\end{cor}

If in addition $Y_1,\dots,Y_s$ are top--dimensional irreducible components of $Y$ and $Y_j$ has scheme--theoretic multiplicity $m_j$ in $Y$ (in other words at level of ideal sheaves we have $\mathcal I_Y = \mathcal I_{Y_1}^{m_1} \cap \dots \cap \mathcal I_{Y_s}^{m_s}$), we can prove it holds $$ w_{\rm CM}(Y,X,L) = \sum_{j=1}^s w_{\rm CM}(Y_j,X,L)\,m_j^{n-d-1},$$ whence by corollary \ref{cor::fut_exp_intro} we get

\begin{cor}\label{cor::fut_exp_mult_intro}
For $r$ sufficiently large we have $$ F\left(\tilde X, L_r\right) = F\left(X,L\right) + \frac{\sum_{j=1}^s w_{\rm CW} \left(Y_j,X,L\right)}{r^{n-d-1}} + O\left(\frac{1}{r^{n-d}}\right). $$ 
\end{cor}

In this case, thanks to smoothness hypothesis on fibers of $\pi|_Y$, the $Y_j$'s meet only over $t=0$ and for $t \neq 0$ the fiber $\tilde X_t$ is the blow--up of $X_t$ along the submanifolds $\pi|_{Y_1}^{-1}(t), \dots, \pi|_{Y_s}^{-1}(t)$. Thus the family $(\tilde X/\mathbb C ,L_r)$ is a test configuration for the polarized manifold $$ \left(\tilde M, \beta^*A^r \otimes \mathcal O(-\sum_{j=1}^s m_jE_j) \right),$$ where $E_j$ is the exceptional divisor over $\pi|_{Y_j}^{-1}(1)$ (the latter, endowed with the structure of reduced scheme).

In case $Y$ has relative dimension zero ({\it i.e.} $d=0$, so that $\tilde M$ is the blow--up of $M$ at a bunch of points), corollary \ref{cor::fut_exp_mult_intro} reduces to a result originally due to Stoppa \cite{Sto08} (but see also subsequent paper \cite[proposition 2.13]{Sto08bis}). We notice that methods used by Stoppa are quite different than our ones: in particular he prove his formula by means of a careful study of the geometry of the central fiber of the blown-up test configuration $\tilde X$.

At this point we recall partial proof of the Yau--Tian--Donaldson conjecture due to Donaldson 

\begin{thm}[Donaldson, \cite{Don02}]\label{thm::cscK_imp_K-stab_intro}
If the polarized manifold $(M,A)$ is K--unstable ({\it i.e.} there exists a test configuration $(X/\mathbb C,L)$ for $(M,A)$ such that $F(X,L) < 0$), then there are no cscK metrics in $c_1(L)$. 
\end{thm} 
Thus, if $(X/\mathbb C,L)$ is a destabilizing test configuration for $(M,L)$, taking a submanifold $N \subset M$ and letting $Y$ be the closure of the trajectory of $N$ under the $\mathbb C^\times$--action on $X$, by discussion above and corollary \ref{cor::fut_exp_intro} we know that the blow--up $\tilde M$ of $M$ with center $N$ admits no cscK metrics in the class $c_1\left( \beta^*A^r \otimes \mathcal O(-E) \right)$ for $r$ sufficiently large. In other words the blow--up of a K-unstable polarized manifolds remains unstable with polarizations which make the exceptional divisor small enough. More interestingly, if $(M,L)$ is K--semistable ({\it i.e.} not unstable) and $(X/\mathbb C,L)$ is now a non--trivial test configuration with $F(X,L)=0$ (which exists whenever there is a non--trivial $\mathbb C^\times$--action on $M$), then the instability of $\left(\tilde M, \beta^* A^r \otimes \mathcal O(-E) \right)$ for $r \gg 0$ is implied by the condition $w_{\rm CW}(Y,X,L)<0$. The latter being exactly the definition of Chow--instability of $N \subset M$ if $X=M\times \mathbb C$ is a product configuration. In this circle of ideas thanks to theorem \ref{thm::cscK_imp_K-stab_intro} we prove the following 

\begin{thm}\label{thm::no_cscK_intro}
Let $(M,A)$ be a polarized manifold admitting a cscK metric with K\"ahler class $c_1(A) \in H^{1,1}(M,\mathbb C)$. Let $ N_1, \dots, N_s \subset M$ be pairwise disjoint submanifolds of co--dimension greater than two.
Let $\beta: \tilde M \to M$ be the blow-up of $M$ along $N_1 \cup \dots \cup N_s$ with $E_j$ the exceptional divisor over $N_j$. For $m_1, \dots, m_s \in \mathbb N$ consider the sub--scheme $N$ cut out by the ideal sheaf $\mathcal I_N = \mathcal I_{N_1}^{m_1} \cap \dots \cap \mathcal I_{N_s}^{m_s}$.

If $N \subset M$ is Chow--unstable then the class $$c_1\left( \beta^*A^r\otimes \mathcal O( - \sum_{j=1}^s m_j E_j) \right) \in H^{1,1}(\tilde M,\mathbb C)$$ contains no cscK metrics for $r \gg 0$.
\end{thm}

Thanks to the well known geometric characterization of Chow stability of configuration of linear subspaces of $\mathbb P^n$ (see for example \cite{MumFogKir94,Dol03} or \cite{Muk03} for points), theorem \ref{thm::no_cscK_intro} gives infinitely many examples of polarized manifolds with no cscK metrics. For examples involving  blow--up of points we refer to \cite{Sto08}, here we give an example with higher dimensional center of blow--up.

\begin{ex}[Projective space blown--up at a pair of skew linear subspaces]
From Nadel \cite[example 6.4]{Nad90} we know that the blow--up of $\mathbb P^{2r+1}$ along a pair of skew $r$-dimensional linear subspaces $L_1$, $L_2$ admits the K\"ahler--Einstein metric. Thus we have the cscK metric in the canonical class $c_1 \left( \mathcal O \left((2r+2)H -r(E_1+E_2)\right)\right)$, where $H$ is the pull--back of the hyperplane class and $E_1$, $E_2$ are the exceptional divisors respectively over $L_1$ and $L_2$. On the other hand, by geometric criterion aforementioned we know that the sub-scheme $L$ cut by $\mathcal I_L = \mathcal I_{L_1}^{m_1} \cap \mathcal I_{L_2}^{m_2}$ is Chow-unstable if $m_1 \neq m_2$. Thus we have no cscK metrics in the classes $c_1\left(\mathcal O \left( kH -m_1E_1-m_2E_2 \right)\right)$ with $m_1 \neq m_2$ and $k$ sufficiently large.
\end{ex} 

Theorem \ref{thm::cscK_imp_K-stab_intro} has an analogous for eK metrics due to Sz\'ekelyhidi \cite{Sze06}, where one has to look to a restricted class of test configuration to get the relevant instability (or equivalently one can correct the Futaki invariant with an additional term). Thanks to this result, with a little additional effort, from corollary \ref{cor::fut_exp_intro} we can get a non-existence theorem for eK metrics on blown--up manifolds. 

\begin{thm}\label{thm::no_eK_intro}
Let $(M,A)$ be a polarized manifold admitting an eK metric with K\"ahler class $c_1(A) \in H^{1,1}(M,\mathbb C)$. Let $ N_1, \dots, N_s \subset M$ be pairwise disjoint submanifolds of co--dimension greater than two and assume there exists an extremal vector field of the class $c_1(A)$ which is tangent to each $N_j$.
Let $\beta: \tilde M \to M$ be the blow-up of $M$ along $N_1 \cup \dots \cup N_s$ with $E_j$ the exceptional divisor over $N_j$. For $m_1, \dots, m_s \in \mathbb N$ consider the sub--scheme $N$ cut out by the ideal sheaf $\mathcal I_N = \mathcal I_{N_1}^{m_1} \cap \dots \cap \mathcal I_{N_s}^{m_s}$.

If $N \subset M$ is relatively Chow--unstable then the class $$c_1\left( \beta^*A^r\otimes \mathcal O( - \sum_{j=1}^s m_j E_j) \right) \in H^{1,1}(\tilde M,\mathbb C)$$ contains no eK metrics for $r \gg 0$.
\end{thm}

Here by {\it relatively Chow--unstable} we mean that, keeping notations as above, $X=M \times \mathbb C$ is a product configuration and $N \subset M$ is unstable with respect to a one--parameter subgroup of $Z_{{\rm Aut}(M)}(T)/T$, where $T \subset {\rm Aut}(M)_N$  is a fixed maximal torus of the stabilizer ${\rm Aut}(M)_N= \left\{ g \in {\rm Aut}(M) \,|\, g(N) = N \right\}$ and $Z_{{\rm Aut}(M)}(T)$ is the identity component of the centralizer of $T$ in ${\rm Aut}(M)$. In general relative stability is weaker than stability, but they are equivalent when $T$ is trivial. A geometric characterization of relative stability of configuration of points in $\mathbb P^n$ will be given in section \ref{sec::rel_stab_P^n}. In the very special case of $\mathbb P^2$ we have the following

\begin{thm}\label{thm::rel_stab_P2}
For fixed points $p_1,\dots,p_s \in \mathbb P^2$ and multiplicities $m_1,\dots,m_s \in \mathbb N$, let $N$ be the sub-scheme (configuration of points) of $\mathbb P^2$ cut by the ideal $\mathcal I_N = \mathcal I_{p_1}^{m_1} \cap \dots \cap \mathcal I_{p_s}^{m_s}$. We have the following cases 
\begin{enumerate}\renewcommand{\theenumi}{\roman{enumi}}
\item \label{thm::rel_stab_P2;few_points} if $s \leq 2$ or $s=3$ and $p_1,p_2,p_3$ are non--aligned, then $N$ is relatively Chow--stable,
\item \label{thm::rel_stab_P2;line} if $s \geq 3$ and $p_j$'s lie on a line $L$ , then $N$ is relatively Chow--stable if and only if it is Chow--stable as a sub-scheme of $L$,
\item \label{thm::rel_stab_P2;line+point} if $s \geq 4$ and $p_2,\dots,p_s$ lie on a line $L$, then $N$ is relatively Chow--stable if and only if the sub-scheme $N'$ cut by the ideal $\mathcal I_{N'} = \mathcal I_{p_2}^{m_2} \cap \dots \cap \mathcal I_{p_s}^{m_s}$ it is Chow--stable as a sub-scheme of $L$,
\item \label{thm::rel_stab_P2;many_nonal_points} if $s \geq 4$ and four points among $p_j$'s are three by three non-aligned, then $N$ is relatively Chow--stable if and only if it is Chow--stable.  
\end{enumerate}  
\end{thm}

As is known since Calabi seminal paper \cite{Cal82}, the existence of eK metric on blow--up of $\mathbb P^2$ at one point is unobstructed, in accordance with theorems \ref{thm::rel_stab_P2} and \ref{thm::no_eK_intro}. More interestingly, for the blow--up of $\mathbb P^2$ at two points, by Arezzo--Pacard--Singer \cite{ArePacSin07}, Chen--LeBrun--Webber \cite{CheLeBWeb07} and He \cite{He07} the existence of eK metrics in classes making small the exceptional divisors and in all classes making the exceptional divisors with the same volume (also called bilaterally symmetric classes) is established. Moreover, for the blow-up of $\mathbb P^2$ at three non--aligned points we know there exist eK metrics in two set of classes: the canonical class and in classes nearby by Tian--Yau \cite{TiaYau87} and LeBrun--Simanca \cite{LeBSim94}, and in a class with exceptional divisors with the same volume grater than in the canonical class and in classes nearby thanks to Arezzo--Pacard--Singer \cite{ArePacSin07}. In this cases our theorems \ref{thm::rel_stab_P2} and \ref{thm::no_eK_intro} give no obstruction providing evidence to the existence of eK metric on each K\"ahler class.   

\begin{ex}[$\mathbb P^2$ blown--up at aligned points]
Consider the blow--up of $\mathbb P^2$ at $s \geq 3$ points lying on a line $L \subset \mathbb P^2$. Thanks to theorems \ref{thm::rel_stab_P2} and \ref{thm::no_eK_intro} we conclude that there are no eK metrics in the Chern class of the line bundle $\mathcal O \left(kH-\sum_{j=1}^s m_jE_j\right)$ if there is a $m_i$ greater than the sum of other ones and $k$ sufficiently large, being $E_j$'s the exceptional divisors on the blown--up points and $H$ the pull--back of the hyperplane class of $\mathbb P^2$.  
\end{ex}

\begin{ex}[$\mathbb P^2$ blown--up at non--aligned points]
In the situation of previous example, if we blow--up another point $p_0 \notin L$ then the sub-scheme of $L$ given by points $p_1,\dots,p_s$ with multiplicities $m_1,\dots,m_s$ is unchanged then we conclude that there are no eK metrics in the first Chern class of $\mathcal O \left(kH-\sum_{j=0}^s m_jE_j\right)$ if there is a $1 \leq i \leq s $ such that $m_i > \sum_{1 \leq j \leq s, j \neq i}m_j$ and $k \gg 0$. In particular $m_0$ is completely arbitrary. On the other hand, if we blow--up $\mathbb P^2$ at $s \geq 4$ points and four of them are three by three not aligned, then relative Chow--stability is the classical one and we have no eK-metrics if more than $\frac{2}{3}$ of points counted with respective multiplicity are aligned.   
\end{ex}

The example above is in fact a weak (due to some border--line case still open) converse of \cite[proposition 8.1]{ArePacSin07}, where are given sufficient conditions on the multiplicities $m_j$'s for the existence of eK metrics in the considered classes. 

\begin{ex}[$\mathbb P^3$ blown--up at three skew lines]
The stabilizer of three skew lines $L_1$, $L_2$, $L_3$ in $\mathbb P^3$ is $SL(2)$, thus the blow--up at that lines can {\it a priori} have eK metrics with non constant scalar curvature in some K\"ahler class. On the other hand, a direct computation of $w_{\rm CW}$ for each line (see remark \ref{rem::CW_weight}) shows that the scheme $N$ cut by the ideal sheaf $\mathcal I_N = \mathcal I_{L_1}^{m_1} \cap \mathcal I_{L_2}^{m_2} \cap \mathcal I_{L_3}^{m_3}$ is relatively Chow unstable if one of the multiplicities $\{m_1, m_2, m_3\}$ is grater than the sum of the others. In this case, by theorem \ref{thm::no_eK_intro} there are no eK metrics in the classes $c_1\left(\mathcal O \left( kH -m_1E_1 - m_2E_2 -m_3E_3 \right)\right)$ with $k\gg 0$, where $E_j$ is the exceptional divisor over $L_j$ and $H$ is the pull--back of the hyperplane class of $\mathbb P^3$.
\end{ex}

\subsection*{Acknowledgments}
The author would like to thank R. P. Thomas for the starting idea of studying the relation between relative K-stability of the blow--up of $\mathbb P^n$ at points and relative Chow--stability of the center of blow--up. It is a great pleasure to thank G. Tian, whose kind interest and support made possible the visit of the author to the mathematics department of Princeton University, where part of this work has been written. Finally the author want to express his deep gratitude to C. Arezzo for many enlightening discussions and precious encouragements.


\section{Relative GIT stability}

The idea of \emph{relative (GIT) stability} is buried in a number of papers \cite{Kir84,Mab04}. A clear treatment related to stability of polarized manifold is given by Sz\'ekelyhidi \cite{Sze06}. Following his definitions, here we recall some elementary facts.

Let $V$ be a finite dimensional $\mathbb C$--vector space acted on linearly by a reductive algebraic group $G$. According to GIT \cite{Mum77,Dol03,Tho06} , a point $v \in V$ is called
\begin{itemize}
\item \emph{unstable} if $0 \in \overline {G \cdot v}$,
\item \emph{semistable} if $0 \notin \overline {G \cdot v}$,
\item \emph{polystable} if $G \cdot v$ is closed in $V$ and $0 \notin G \cdot v$,
\item \emph{stable} if it is polystable and the stabilizer $G_v$ is finite.
\end {itemize}

Now let $H \subset G$ a reductive subgroup. Obviously the action of
$G$ induces an action of $H$ and a (semi/poly)stable point $v$
with respect to $G$ is (semi/poly)stable with respect to $H$. By
converse, the stability with respect a sufficiently wide class of
subgroups implies the stability with respect $G$. Indeed it holds
the following
\begin{thm}[Hilbert--Mumford criterion]
A point $v \in V$ is (semi/poly)stable if and only if it is
(semi/poly)stable with respect to all non--trivial
one--parameter subgroups of $G$.
\end{thm}

Now let $\lambda$ be a one-parameter subgroup of $G$. Each non--zero $v \in V$ determines a line $L=[v] \in \mathbb P(V)$ and $L_0=\lim_{t \to 0} \lambda(t) \cdot L$ is $\lambda$-invariant. The weight of the action of $\lambda$ on $L_0 \subset V$  is denoted by $\mu(v,\lambda)$ and it is called the \emph{Mumford weight} of $v$ with respect to $\lambda$. Thus, by Hilbert-Mumford criterion, $v$ is (semi)stable if and only if $\mu(v,\lambda) < 0$ (resp. $\leq 0$) for all one--parameter subgroups $\lambda$ of $G$.

Next, in order to define relative stability of a point $v \in V$ consider the stabilizer $G_{[v]}=\{g \in G \,|\, g \cdot v \in \mathbb C v\}$ of $[v]$ with respect to the induced action on $\mathbb P(V)$ and choose a maximal torus $ T \subseteq G_{[v]}$, so that $ T \simeq (\mathbb C^\times)^r $ for some integer $r\geq 0$. Let $Z_G(T)$ be the identity component of the centralizer of $T$ in G and consider the quotient group $G_T = Z_G(T)/T$. Choosing a split of the exact sequence $ 0 \to T \to Z_G(T) \to G_T \to 0$ gives an action of $G_T$ on $V$ by restriction of the original one of $G$.

\begin{defn}\label{defn::relative stability}
A point $v \in V$ is called \emph{relatively (un/semi/poly)stable} if it is (un/semi/poly)stable with respect to $G_T$.
\end{defn}

Actually the inclusion $\iota : G_T \hookrightarrow G$ is not canonically defined, thus we have to show that definition \ref{defn::relative stability} is independent of $\iota$. To this end let $\kappa: G_T \to Z_G(T)$ be another split of the sequence above. Since $\iota$ and $\kappa$ are right inverses of the projection $Z_G(T) \to G_T$ then $\tau(g) = \kappa(g)\iota(g)^{-1}$ defines an homomorphism $\tau: G_T \to T$. Thus the orbit $\kappa(G_P)\cdot v$ is just a ``translation'' of the orbit $\iota(G_P) \cdot v$ via $\tau$.    

\begin{rem}
A stable point for the $G$--action is relatively stable as well. Indeed, in this case $T$ is trivial and $G_T=Z_G(T)=G$. On the other hand, relatively unstable points are unstable in absolute sense.
\end{rem}

\begin{rem}
As in the non-relative case, in presence of a polarized variety $(X,L)$ acted on by $G$, the relative stability of $p \in X$ is defined looking at the orbit of a nonzero $\ell \in L_p$ under the action of $G_T$, being $T \subseteq G_p$ a maximal torus of the stabilizer of $p$ .   
\end{rem}

\subsection{Relative stability of configurations of points}\label{sec::rel_stab_P^n}
The aim of this section is to give a geometric criterion for relative Chow--stability of configuration of points in the projective space $\mathbb P^n$. For convenience of the reader and future reference we start by recalling a similar well--known result in the absolute (i.e. non relative) Chow--stability. For proofs and more details see \cite{MumFogKir94,Dol03,Muk03}.

By {\it configuration of points} here we mean an element $P=(p_1,\ldots,p_m) \in (\mathbb P^n)^m$. 

\begin{rem}
To $P$ we can associate an ideal sheaf $\mathcal I_P \subset \mathcal O_{\mathbb P^n}$ as follows. Perhaps changing the order of $p_j$'s we can suppose that $\{p_1, \dots, p_s\}$ is the maximal set of distinct points among $p_j$'s. Denoted by $m_j$ the multiplicity of $p_j$ in $P$ for each $1\leq j\leq s$, we set $\mathcal I_P = \mathcal I_{p_1}^{m_1} \cap \dots \cap \mathcal I_{p_s}^{m_s}$. As well known the Chow--stability of the subscheme of $\mathbb P^n$ defined by the ideal sheaf $\mathcal I_P$ is equivalent to the GIT stability of the $m$-form $f_P \in \mathbb C[x_0,\dots,x_x]$ defined in the following.
\end{rem}

To each $p_j = (p_j^0:\ldots:p_j^n)$ we associate the linear form $l_j (x) = \sum_{i=0}^n p_j^i x_i$ (the so--called Chow form of $p_j$) and then consider the product $$ f_P(x) = \Pi_{j=1}^m l_j(x).$$

Now let $V=\mathbb C[x_0,\ldots,x_n]_m$ be the space of forms of degree $m$ endowed with the action of $G=SL(n+1)$ defined by $$ (g \cdot f)(x) = f(g^{-1}x), $$ for each $g \in G$. Since $f_P \in V$ we give the following

\begin{defn}
The configuration $P$ is called (semi)stable if $f_P \in V$ is.
\end{defn}

The geometric meaning of stability just defined is given by the following

\begin{thm}\label{thm::abs_sta}
Let $P \in \mathbb (P^n)^m$ a configuration of points of $\mathbb P^n$. $P$ is semistable if and only if for every proper linear subspace $E \subset \mathbb P^n$ we have $$ \# \{j \,|\, p_j \in E \}\leq \frac{\dim E +1}{n+1}m.$$ P is stable if and only if the strict inequality holds.
\end{thm}

\begin{proof} See \cite[Proposition 7.27]{Muk03} or \cite[Theorem 11.2]{Dol03}. \end{proof}

Now we turn to a relative version of the previous results.

\begin{defn}
$P$ is called relatively (semi/poly)stable if $f_P \in V$ is.
\end{defn}

To give geometric conditions characterizing relative stability of $P$, let $G_P \subseteq G$ be the identity component of the stabilizer of the configuration $P$ and let $\Lambda \subseteq \mathbb P^n$ be the $G_P$-invariant subspace spanned by points $p_1,\ldots,p_m$. Consider the action of $G_P$ on $\Lambda$ and let $F$ be the fixed points locus. $F$ is a union of subspaces, then $\Lambda$ is decomposable in a sum of pointwise $G_P$-fixed subspaces. The following lemma characterize such decomposition and it is crucial to state the condition of relative stability.

\begin{lem}\label{lem::deco}
There exists a decomposition $$\Lambda=\Lambda_1+\ldots+\Lambda_s$$ with $\{p_1,\dots,p_m\} \cap \Lambda_j \neq \emptyset $ such that
\begin{itemize}
\item $\Lambda_j \cap \sum_{\ell\neq j} \Lambda_{\ell} = \emptyset$, \emph{(orthogonality)}
\item If $\Lambda=\Lambda'_1+\ldots+\Lambda'_{s'}$ satisfies conditions above, then $s'\leq s$. \emph{(irreducibility)}
\end{itemize}
A decomposition as above is unique up to the order of summands.
\end{lem}

\begin{proof}
We prove the statement by constructing such a decomposition.

Let $d-1$ be the dimension of $\Lambda$. By definition of $\Lambda$ it is possible to choose $d$ points linear independent among $p_1,\ldots,p_m$. Without loss we can suppose that $p_1,\ldots,p_d$ are independent. Furthermore, with a suitable choice of projective coordinates, we can also suppose $p_j=e_{j-1}$ (the canonical points of $\mathbb P^n$ in the fixed coordinates) for $1 \leq j \leq d$.

To each point $p_j = (p_j^0: \ldots :p_j^n) \in \{p_1,\ldots,p_m\}$ we associate the subspace $$S_j = \sum_{i \mbox{ s.t. } p_j^i\neq 0} \{ e_i\}.$$ Clearly $p_j \in S_j$ and the number of non--zero coefficient of $p_j$ with respect to the canonical basis equals $\dim S_j$. In particular $S_j=\{ e_{j-1} \}$ for $j \in \{1,\dots,d\}$. Now we introduce on the set of subspaces $\Sigma =\{S_1,\dots,S_m\}$ the equivalence relation defined by $$ S_i \sim S_j \quad \iff \quad S_i \cap S_j \neq \emptyset$$ and we associate to each class $C_i \in \{C_1,\dots,C_s\} = \Sigma/\sim$ the subspace $$\Lambda_i = \sum_{S_j \in C_i} S_j.$$ By construction the decomposition $$\Lambda = \Lambda_1 + \dots + \Lambda_s$$ satisfies orthogonality and irreducibility condition.
\end{proof}

\begin{ex}[Plane configurations of points]\label{ex::plane_configs}
Dealing with points on the plane only few different cases can occur. Let $P \in (\mathbb P^2)^m$ and $\Lambda$ be the span of points of $P$. One has the following cases

\vspace{3mm}
\begin{tabular}{ll}
$\Lambda = \{p\}$                & if $P$ is supported at the point $p$.\\
$\Lambda = \{p\}+\{q\}$          & if $P$ is supported at two distinct points $p$, $q$.\\
$\Lambda = \{p\}+\{q\}+\{r\}$    & if $P$ is supported at three non--aligned points $p$, $q$, $r$. \\
$\Lambda = L$                    & if $P$ is supported at least at three\\
                                 & distinct points of the line $L$.\\
$\Lambda = \{p\}+L$              & if $P$ is supported at the point $p \notin L$ and \\
                                 & at least three distinct points of the line $L$.\\
$\Lambda = \mathbb P^2$          & if $P$ is supported at least at four points,\\
                                 & three by three not aligned.
\end{tabular}
\end{ex}

Now we are in position to state and prove the main result of this section

\begin{thm}\label{thm::rel_sta}
Let $P \in \mathbb (\mathbb P^n)^m$ be a configuration of points in $\mathbb P^n$. Let $\Lambda = \Lambda_1 + \dots + \Lambda_s$ be the subspace spanned by points of $P$ with the decomposition of lemma \ref{lem::deco}. For each $j \in \{1,\dots,s\}$ let $P_j$ be the configuration of points of $P$ contained in $\Lambda_j$.

The configuration $P$ is relatively (semi/poly)stable if and only if each configuration $P_j$ is (semi/poly)stable in $\Lambda_j$.
\end{thm}

\begin{proof}
First of all we have to determine the (identity component of the) stabilizer $G_P$ of $P$ in $G=SL(n+1)$. To this end is useful to make some assumption on points of $P=(p_1,\dots,p_m)$. As in the proof of lemma \ref{lem::deco} we can suppose without loss that $p_j=e_{j-1}$ for $1 \leq j \leq d = \dim \Lambda +1$. Moreover, since each $\Lambda_j$ is spanned by some $e_i$'s, perhaps after some exchanges, we can also assume \begin{eqnarray*} \Lambda_1 & = & \{e_0\} + \dots + \{e_{d_1 - 1}\}, \\ \Lambda_2 & = & \{e_{d_1}\} + \dots + \{e_{d_1 + d_2 -1}\}, \\ & \vdots & \\ \Lambda_s & = & \{e_{d_1+ \dots + d_{s-1} }\} + \dots + \{e_{d_1+ \dots + d_s-1}\}, \end{eqnarray*}
where $d_j = \dim \Lambda_j +1$, with $\sum_{j=1}^s d_j = d$. By irreducibility property of decomposition $\Lambda = \Lambda_1 + \dots + \Lambda_s$, if $g \in G_P$ then $g$ fixes each point of $\Lambda_j$ for any $j \in \{1,\dots,s\}$. Thus, in the basis $\{e_0,\dots,e_n\}$ any element $g \in G_P$ is represented by a matrix of the form $$g =
\left(
\begin{array}{ccc|p{5mm}|}
\cline{1-1} \cline{4-4}
\multicolumn{1}{|c|}{\lambda_1 I_1} &        &               &     \\
\cline{1-1}
                                    & \ddots &               & \multicolumn{1}{c|}{C}  \\
\cline{3-3}
                                    &        & \multicolumn{1}{|c|}{\lambda_s I_s} &     \\
\cline{3-4}
                                    &        &               & \multicolumn{1}{c|}{B} \\
\cline{4-4}
\end{array}
\right)
$$
where $\lambda_j \in \mathbb C^\times$, $I_j$ is the identity matrix
of rank $d_j$, $ C \in {\rm Mat}(d\times
(n-d+1),\mathbb C)$ and $B \in GL(n-d+1)$ satisfies the
condition $$ \det(B) = \frac{1}{\lambda_1^{d_1} \dots
\lambda_s^{d_s}}.$$ With the most natural choice of the maximal torus $T \subseteq G_P$ a element $t \in T$ is represented by a matrix of the form
$$t =
\left(
\begin{array}{cccccc}
\cline{1-1}
\multicolumn{1}{|c|}{\lambda_1 I_1} &        &               &    &    & \\
\cline{1-1}
                                    & \ddots &               &    &    & \\
\cline{3-3}
                                    &        & \multicolumn{1}{|c|}{\lambda_s I_s} &    &    & \\
\cline{3-6}
                                    &        &               & \multicolumn{1}{|c}{\lambda_{s+1}} & & \multicolumn{1}{c|}{} \\
                                    &        &               & \multicolumn{1}{|c}{} & \ddots & \multicolumn{1}{c|}{} \\
                                    &        &               & \multicolumn{1}{|c}{} & & \multicolumn{1}{c|}{\lambda_{\rho+1}} \\
\cline{4-6}
\end{array}
\right)
$$
where $I_j$ are as above, $\lambda_j \in \mathbb C^\times$ with the
condition $$ \lambda_{s+1}\dots \lambda_{\rho+1} =
\frac{1}{\lambda_1^{d_1} \dots \lambda_s^{d_s} },$$ and $\rho= n-d+s
$ is the dimension of $T$. Thus, a generic element $g \in Z_G(T)$ of the centralizer of $T$ in $G$ is represented by
\begin{equation}\label{eqn::el_centr}
g =
\left(
\begin{array}{cccccc}
\cline{1-1}
\multicolumn{1}{|c|}{A_1} &        &               &    &    & \\
\cline{1-1}
                                    & \ddots &               &    &    & \\
\cline{3-3}
                                    &        & \multicolumn{1}{|c|}{A_s} &    &    & \\
\cline{3-6}
                                    &        &               & \multicolumn{1}{|c}{\mu_{s+1}} & & \multicolumn{1}{c|}{} \\
                                    &        &               & \multicolumn{1}{|c}{} & \ddots & \multicolumn{1}{c|}{} \\
                                    &        &               & \multicolumn{1}{|c}{} & & \multicolumn{1}{c|}{\mu_{\rho+1}} \\
\cline{4-6}
\end{array}
\right)
\end{equation}
where $\mu_j \in \mathbb C^\times$ and $A_j \in {\rm GL}(d_j)$ satisfy the condition \begin{equation}\label{eqn::SL_Ak} \det(A_1) \dots \det(A_s) = \frac{1}{\mu_{s+1}\dots\mu_{\rho+1}}.\end{equation}
Now we define $\iota: G_T=Z_G(T)/T \to G$ by $$ \iota(gT) =
\left(
\begin{array}{cccp{5mm}}
\cline{1-1} 
\multicolumn{1}{|c|}{B_1} &        &               &     \\
\cline{1-1}
                                    & \ddots &               &  \\
\cline{3-3}
                                    &        & \multicolumn{1}{|c|}{B_s} &     \\
\cline{3-4}
                                    &        &               & \multicolumn{1}{|c|}{I} \\
\cline{4-4}
\end{array}
\right)
$$
where $I$ is the identity matrix of rank $\rho-s+1$ and $B_j = (\det A_j)^{-1}A_j \in SL(d_j)$. 

Under our assumptions the homogeneous form $f_P \in \mathbb
C[x_0,\dots,x_n]_m$ splits $$ f_P = f_1 \dots f_s$$ where $f_j \in
\mathbb C[x_{d_1 + \dots + d_{j-1}}, \dots, x_{d_1 + \dots + d_j
-1}]$. Thus the action of an element $g \in Z_G(T)$ of the form
\eqref{eqn::el_centr} splits into the action of the blocks $B_j$ on
$f_j$: \begin{equation}\label{eqn::split_act} g \cdot f_P = (B_1
\cdot f_1) \dots (B_s \cdot f_s). \end{equation} 

After all these reductions we are ready to show the ties between the absolute stability of $f_j$'s and the relative stability of $f_P$. Let us suppose that $f_j$ is unstable for the action of $SL(d_j)$. This means that the null form $0$ is contained in the closure of the orbit $SL(d_j) \cdot f_j$. Since $SL(d_j) \subseteq \iota(G_T)$ also $f_P$ is unstable. By converse, if every $f_j$ is semistable then the closure of each orbit $SL(d_j) \cdot f_j$ does not contain the null form $0$. This implies that the orbit $\iota(G_T) \cdot f_P$ of $f_P$ does not contain the null form $0$ and $f_P$ is semistable. Next we pass to polystability. By splitting \eqref{eqn::split_act} follows that the orbit $\iota(G_T) \cdot f_P$ is the pointwise product of orbits $SL(d_j) \cdot f_j$, thus it is closed if and only if each orbit $SL(d_j) \cdot f_j$ is closed. This implies that $f_P$ is relatively polystable if and only if each $f_j$ is polystable for the $SL(d_j)$--action. Finally, since the stabilizer of $f_P$ in $\iota(G_T)$ is notingh but the direct product of stabilizers in $SL(d_j)$  of each $f_j$, then $f_P$ is relatively stable if and only if each $f_j$ is stable under the action of $SL(d_j)$.
\end{proof}

\begin{proof}[proof of theorem \ref{thm::rel_stab_P2}]
It is a corollary of theorem \ref{thm::rel_sta} via example \ref{ex::plane_configs}. In case \ref{thm::rel_stab_P2;few_points}. each $\Lambda_j$ is a point, thus $N$ restricted to it is always stable. In cases \ref{thm::rel_stab_P2;line}. and \ref{thm::rel_stab_P2;line+point}. $\Lambda$ is either the line $L$ or the sum of $L$ with a point; in any cases the stability of $N$ restricted to $L$ is equivalent to the relative stability of $N$. Finally, in case \ref{thm::rel_stab_P2;many_nonal_points}. $\Lambda=\mathbb P^2$, thus relative stability and stability coincide. 
\end{proof}


\section{CM and CW--lines}\label{sec::CM&CW}

In this section we recall basic facts on CM--line, originally defined in \cite{PauTia06} (for a nice review and some interesting positivity results see also \cite{FinRos06}). Moreover we introduce the CW--line, related to Chow--stability of subschemes.

Let $\pi: X \to B$ be a flat morphism of projective varieties with $B$ irreducible and relative dimension $\dim(X/B) = n$ and let $L$ be a relatively ample line bundle $X$. For each $k$ consider the coherent sheaf $\pi_!( L^k )$ on $B$. By standard theory the Hilbert polynomial $\chi(X_b,L_b^k)$ of the fiber of $X_b=\pi^{-1}(b)$ equipped with the polarization $L_b=L|_{X_b}$ is independent of $b \in B$ and is given by \begin{equation}\label{eq::rk_av} {\rm rank}\, \pi_! (L^k) = \sum_{j=0}^n a_jk^{n-j}, \end{equation} where $a_j \in \mathbb Q$.   

Analogously, thanks to Knudsen and Mumford results \cite{KnuMum76} we have the \textit{polynomial expansion} \begin{equation}\label{eq::det_av} \det \pi_! (L^k) = \bigotimes_{j=0}^{n+1} \nu_j^{k^{n+1-j}},\end{equation} where $\nu_j$'s are $\mathbb Q$-line bundles on $B$. The relevance of $\det \pi_! (L^k)$ to our aims rely on the canonical isomorphism $$ \left( \det \pi_! (L^k) \right)_b \simeq \det H^0(X_b,L_b^k), $$ for each $b \in B$ and $k \gg 0$. 
   
\begin{defn}[\cite{PauTia06}]\label{def::CMline}
The CM--line associated to the polarized family $(X/B,L)$ is the $\mathbb Q$-line bundle on $B$ given by $$ \lambda_{\rm CM}(X/B,L) = \left( \nu_0^{-a_1} \otimes \nu_1^{a_0}\right)^{\frac{1}{a_0^2}}. $$ 
\end{defn}

We remark that the CM--line just defined is a rational multiple of the original one in \cite{PauTia06}. On the other hand the CM--line as defined in \ref{def::CMline} enjoy two nice properties. First it is homogeneous of degree zero as function of $L$, in other words $$ \lambda_{\rm CM}(X/B,L^r) = \lambda_{\rm CM}(X/B,L) $$ for all $r$. Second, in presence of a $\mathbb C^\times$-action, the weight over the invariant points of $B$ is the generalized Futaki invariant defined by Donaldson \cite{Don02}. More precisely, suppose given $\mathbb C^\times$-actions on $X$ and $B$ making $\pi$ an equivariant map and choose a linearization on $L$. From these data, we get a linearization on $\lambda_{\rm CM}(X/B,L)$. Is not difficult to see that such linearization is independent of the one on $L$. If $b_0 \in B$ is a fixed point let $F_1(b_0)$ be the weight of the action on the fiber $(\lambda_{\rm CM}(X/B,L))_{b_0}$. On the other hand, taking the fiber of $\pi$ over $b_0$, we get a polarized scheme $(X_{b_0},L_{b_0})$ endowed with a $\mathbb C^\times$-action and a linearization on $L_{b_0}$. Denoted by $w(X_{b_0},L_{b_0}^k)$ the weight of the induced action on $\det H^0(X_{b_0},L_{b_0}^k)$, we have the following 

\begin{thm}[Paul--Tian \cite{PauTia06}]\label{thm::PT_Fut} $$ \frac{w(X_{b_0},L_{b_0}^k)}{k \, \chi(X_{b_0},L_{b_0}^k)} = F_0(b_0) + F_1(b_0)k^{-1} + O(k^{-2}),\qquad k \gg 0.$$ \end{thm}

\begin{rem}
With our signs conventions, $(X_{b_0},L_{b_0})$ is ${\rm CM}$-unstable w.r.t. a given $\mathbb C^\times$--action on $(X/B,L)$ if the Futaki invariant $F_1(b_0)$ is {\it positive}. 
\end{rem}

\bigskip

Now we pass to define the CW--line. Let $Y$ be a closed subscheme of $X$ flat over $B$. In other words we require that $\pi|_Y:Y \to B$ is flat or equivalently, that $\mathcal I_Y$ is flat over $B$, being $\mathcal I_Y \subset \mathcal O_X$ the ideal sheaf of $Y$. By \cite[Proposition 2.1]{Mum77} there is a polynomial expansion \begin{equation}\label{eq::rk_sub} {\rm rank}\, \pi_! \left( L^h/\mathcal I_Y^k L^h \right) = \sum_{i=0}^d \sum_{j=0}^{n-d} b_{i,j} h^{d-i} k^{n-d-j}, \end{equation} where $b_{i,j} \in \mathbb Q$. Moreover, arguing as in the proof of \cite[Proposition 2.1]{Mum77}, and using Cartier's theorem \cite{Car60} in place of Snapper, we get the following expansion \begin{equation}\label{eq::det_sub} \det \pi_!(L^{h}/\mathcal I_Y^k L^{h}) = \bigotimes_{i=0}^{d+1} \bigotimes_{j=0}^{n-d} \rho_{i,j}^{h^{d+1-i}k^{n-d-j}}, \end{equation} where $\rho_{i,j}$ are fixed $\mathbb Q$-line bundles on $B$ and $d = \dim(Y/B)$ is the relative dimension of $Y$ over $B$. 

\begin{defn}\label{def::CWline}
The CW--line associated to the closed subscheme $Y \subset (X/B,L)$ is the $\mathbb Q$-line bundle on $B$ given by $$ \lambda_{\rm CW}(Y,X/B,L) = \left( \nu_0^{b_{0,1}} \otimes \rho_{0,1}^{-a_0}\right)^{\frac{1}{a_0^2}}. $$ 
\end{defn}

Given $\mathbb C^\times$-actions on $B$ and $X$ making $Y$ invariant and $\pi$ equivariant, after choosing a linearization on $L$ we get an induced linearization on $\lambda_{\rm CW}(Y,X/B,L)$. This linearization is natural, in the sense that it is independent of the one on $L$. Pick a $\mathbb C^\times$-fixed point $b_0 \in B$ and consider the fiber $Y_{b_0}=\pi|_Y^{-1}(b_0)$

\begin{defn}
For each $b\in B$ such that $\lim_{t\to 0} t \cdot b = b_0$, the {\it Chow weight} $ w_{\rm CW}(Y_b,X_b,L_b)$ of $Y_b=\pi|_Y^{-1}(b)$ w.r.t. the given $\mathbb C^\times$-action is the weight of the induced $\mathbb C^\times$-action on the fiber $\left( \lambda_{\rm CW}(Y,X/B,L) \right)_{b_0}$.
\end{defn}

\begin{rem}\label{rem::CW_weight}
The Chow weight just introduced is a generalization of the well-known Chow-Mumford weight of a projective variety $N^d \subset \mathbb P^n$ \cite{Mum77,RosTho07}. To see this, fix a one parameter subgroup of $\alpha: \mathbb C^\times \to SL(n+1)$. With a suitable choice of coordinates we have $\alpha(t) =  {\rm diag}\, (t^{q_0-p},\dots,t^{q_n-p}), $ with $p=\sum_{i=0}^n \frac{q_i}{n+1}$ and $0\leq q_0 \leq \dots \leq q_n$. Fix on the hyperplane bundle $\mathcal O_{\mathbb P^n}(1)$ the linearization of $\alpha$ that induces the $\mathbb C^\times$-action $t \mapsto {\rm diag}\, (t^{-q_0}, \dots, t^{-q_n})$ on $H^0(\mathbb P^n, \mathcal O_{\mathbb P^n}(1))$. Now let $X= \mathbb P^n \times \mathbb C$ acted on diagonally by $\mathbb C^\times$. Clearly the projections on the factors are equivariant and flat, thus the pull-back $L={\rm pr}_1^*\mathcal O_{\mathbb P^n}(1)$ is a linearized line bundle on $X$. Let $\pi: X \to \mathbb C$ be the projection on the second factor. By general theory recalled above, for $k \gg 0$ we have $$ {\rm rank}\, \pi_!(L^k) = \dim H^0(\mathbb P^n, \mathcal O_{\mathbb P^n}(k)) = \frac{k^n}{n!} + O(k^{n-1})$$ and $$ \det \pi_!(L^k) = \nu_0^{k^{n+1}} \otimes O(k^n),$$ whence by $ \left(\det \pi_!(L^k)\right)_0 \simeq H^0(\mathbb P^n, \mathcal O_{\mathbb P^n}(k)),$ we can conclude that the weight of the induced $\mathbb C^\times$-action on the fiber of $\nu_0$ over $t=0$ is given by $ \sum_{i=0}^n -\frac{q_j}{(n+1)!}$. 

 Now consider the closure of the trajectory of $N$ under the action of $\alpha$ $$Y = \overline{ \left\{ (p,t) \in \mathbb P^n \times \mathbb C^\times \,|\, \alpha(t^{-1}) (p) \in N \right\}} \subset X $$ and denote its ideal sheaf with $\mathcal I_Y \subset \mathcal O_X$. By theory above, for $h \gg 0$ we have asymptotic expansions $$ {\rm rank}(L^h/\mathcal I_{Y}^k L^h) = b_0(k)h^d + O(h^{d-1}) $$ and $$ \det \pi_!(L^h/\mathcal I_Y^k L^h) = \rho_0(k)^{h^{d+1}} \otimes O(h^d),$$ thus $b_0(k)$ is the degree of $(N,\mathcal I_N^k) \subset \mathbb P^n$ and the weight $w_0(k)$ of the induced action on the fiber over $t=0$ of the $\mathbb Q$-line bundle $\rho_0(k)$ is the leading coefficient of the polynomial expansion of $\det H^0(\mathbb P^n, \mathcal O_{\mathbb P^n}(h) / (\mathcal I_N^k)_0 \otimes \mathcal O_{\mathbb P^n}(h))$ as $h \gg 0$, where $(\mathcal I_N^k)_0$ is the ideal sheaf of the flat limit $\lim_{t \to 0} \alpha(t)(N,\mathcal O_X/\mathcal I_N^k)$. Since the multiplicity of $[N_{red}]$ in the cycle underlying the scheme $(N,\mathcal I_N^k)$ is  $\binom{n-d+k-1}{k-1}$, by \cite[Lemma 25]{Wan04} we get $$ b_0(k) = \frac{\deg(N)}{d!(n-d)!} k^{n-d} + \frac{\deg(N)}{2\,d!(n-d-2)!}k^{n-d-1} + O(k^{n-d-2}),$$ $$ w_0(k) = \frac{-e(N)}{(d+1)!(n-d)!}k^{n-d} + \frac{-e(N)}{2\,(d+1)!(n-d-2)!}k^{n-d-1} + O(k^{n-d-2}), $$ being $-\frac{e(N)}{(d+1)!}$ the leading coefficient as $h \gg 0$ of the weight of the induced action on $\det H^0(\mathbb P^n, \mathcal O_{\mathbb P^n}(h) / (\mathcal I_N)_0 \otimes \mathcal O_{\mathbb P^n}(h))$.

Thus the weight of the induced action on the fiber over $t=0$ of $\lambda_{\rm CW}(Y,X/\mathbb C,L)=\left( \nu_0^{-b_{0,1}} \otimes \rho_{0,1}^{a_0}\right)^{\frac{1}{a_0^2}} $ is given by $$ w_{\rm CW}(N,\mathbb P^n,\mathcal O_{\mathbb P^n}(1) ) = \frac{n!}{2(d+1)!(n-d-2)!}\left( e(N) - \frac{(d+1) \deg(N)}{n+1} \sum_{i=0}^n q_j \right)$$ which coincides with the Chow-Mumford weight appearing in \cite[Theorem 2.9]{Mum77} up to the factor $\frac{n!}{2(d+1)!(n-d-2)!}$.
\end{rem}

\begin{rem}
With our signs conventions, a variety $N \subset \mathbb P^n$ is Chow-unstable w.r.t. a given $\mathbb C^\times$--action on $\mathbb P^n$ if the Chow--weight $\omega_{\rm CW} (N,\mathbb P^n,\mathcal O_{\mathbb P^n}(1))$ is {\it positive}. 
\end{rem}

\begin{rem}\label{rem::CW_subvar}
We can slightly generalize the situation above and consider $N\subset M$, where $(M,A)$ is a $n$-dimensional polarized manifold endowed with a $\mathbb C^\times$-action $\alpha : \mathbb C^\times \to {\rm Aut}(M)$ that linearizes on $A$, and $N$ is a possibly non-reduced sub-variety. In this case $X=M \times \mathbb C$, $\pi = {\rm pr}_2$ and $L = {\rm pr}_1^*A$. Moreover $$Y = \overline{ \left\{ (p,t) \in M \times \mathbb C^\times \,|\, \alpha(t^{-1}) (p) \in N \right\}} \subset X $$ and $N_0 = \pi|_Y^{-1}(0)$ is the flat limit of $N$ under the action of $\alpha$. Let $a_0(M,A)$ and $a_0(N,A|_N)$ be respectively the leading coefficient of $h^0(M,A^r)$ and $h^0(N,A|_N^r)$ as $r \gg 0$ and let $e(M,A)$, $e(N_0,L|_{N_0})$ be the leading coefficients of polynomial expansions of the total weights of induced actions on $ H^0(M,A^r)$ and $ H^0(N_0,L|_{N_0}^r)$ as $r \gg 0$. Arguing as above we get \begin{equation}\label{eq::Chow_w_asympt} w_{\rm CW}(N,M,L) = \frac{a_0(N,A|_N)}{2(d+1)!(n-d-2)!\,a_0(M,A)} \left( \frac{e(M,A)}{a_0(M,A)} - \frac{e(N_0,L|_{N_0})}{a_0(N,A|_N)}\right).\end{equation}   

Moreover we can get a differential geometric expression for $w_{\rm CW}(N,M,L)$ as follows. Let $v$ be the holomorphic vector field generated by $\alpha$ on $M$ and let $\phi_v$ be the normalized potential of $v$ with respect to a fixed K\"ahler metric $\omega$ in the first Chern class of $A$; in other words $\phi_v$ is the unique solution of the system $$ \left\{ 
\begin{array}{l}
\int_M \phi_v \omega^n = 0 \\
\bar \partial \phi_v + i_v \omega = 0. 
\end{array}
\right.$$     
Since $N_0$ is $\alpha$-invariant, by Riemann-Roch theorem and \cite[proposition 3]{Don05} one has $a_0(M,L) = {\rm vol}(M,\omega)$, $e(M,L)=0$ and $e(N,A|_N) =\int_N \phi_v \frac{\omega^d}{d!}$, whence \begin{equation}\label{eq::CH_diff} w_{\rm CW} (N,M,L)= \frac{-1}{2(n-d-2)!\,{\rm vol}(M,\omega)} \int_{N_0} \phi_v \frac{\omega^d}{d!}, \end{equation} where the integral is over the cycle associated to the scheme $N_0$.

Clearly equation \eqref{eq::CH_diff} defines a linear functional on the space of the holomorphic vector fields (with zeros) on $M$ tangent to $N_0$.
\end{rem}

We conclude this section with the following 

\begin{lem}\label{lem::lAC_mor}
Given a finite collection of closed subschemes $(Y_j, \mathcal I_{Y_j})\subset X$, $1 \leq j \leq s$ of the same dimension $d$, let $Z\subset X$ be the sub-scheme cut out by $\mathcal I_Z = \mathcal I_{Y_1}^{m_1} \cap \dots \cap I_{Y_s}^{m_s}$ for some multiplicities $m_1,\dots,m_s > 0$. We have $$ \lambda_{\rm CW} (Z,X/B,L) = \bigotimes_{j=1}^s \lambda_{\rm CW} (Y_j,X/B,L)^{m_j^{n-d-1}}.$$
\end{lem}

\begin{proof}
We start proving the formula in the case $s=2$ and $m_1=m_2=1$ so that $\mathcal I_Z = \mathcal I_{Y_1} \cap \mathcal I_{Y_2}$. Consider the exact sequence 
\begin{equation}\label{eq::ex_seq_int}
0 \to \mathcal O_X / \mathcal I_{Y_1} \cap \mathcal I_{Y_2} \to \mathcal O_X / \mathcal I_{Y_1} \oplus \mathcal O_X / \mathcal I_{Y_2} \to \mathcal O_X / \mathcal I_{Y_1} + \mathcal I_{Y_2} \to 0,
\end{equation} 
where the third arrow takes the difference of sections of $\mathcal O_X / \mathcal I_{Y_1}$ and $\mathcal O_X / \mathcal I_{Y_2}$. Tensoring by $L^h$, for $h\gg 0$ we get $$ 0 \to L^h / \left(\mathcal I_{Y_1} \cap \mathcal I_{Y_2}\right)^k L^h \to L^h / \mathcal I_{Y_1}^k L^h \oplus L^h / \mathcal I_{Y_2}^k L^h \to L^h / \left( \mathcal I_{Y_1} + \mathcal I_{Y_2} \right)^k L^h \to 0, $$ whence $$ \pi_! \left( L^h / \left(\mathcal I_{Y_1} \cap \mathcal I_{Y_2}\right)^k L^h \right) = \pi_! \left( L^h / \mathcal I_{Y_1}^k L^h\right) + \pi_! \left(L^h / \mathcal I_{Y_2}^k L^h \right) - \pi_! \left( L^h / \left( \mathcal I_{Y_1} + \mathcal I_{Y_2} \right)^k L^h \right).$$ 
Since ${\rm supp}\left(\mathcal I_{Y_1} + \mathcal I_{Y_2}\right)$ has dimension less than $d$, we get \begin{eqnarray*} \det \pi_! \left( L^h / \mathcal I_{Z}^k L^h \right) &=& \det \pi_! \left( L^h / \mathcal I_{Y_1}^k L^h\right) \otimes \det \pi_! \left(L^h / \mathcal I_{Y_2}^k L^h \right) \otimes O(h^d) \\ &=& \bigotimes_{j=0}^{n-d} \rho_{0,j}(Y_1)^{h^{d+1}k^{n-d-j}} \otimes \bigotimes_{j=0}^{n-d} \rho_{0,j}(Y_2)^{h^{d+1}k^{n-d-j}} \otimes O(h^d) \\ &=& \bigotimes_{j=0}^{n-d} \left(\rho_{0,j}(Y_2) \otimes \rho_{0,j}(Y_2) \right)^{h^{d+1}k^{n-d-j}} \otimes O(h^d), \end{eqnarray*} 
and analogously $$ {\rm rank}\,\pi_! \left( L^h / \mathcal I_{Z}^k L^h \right) = \sum_{j=0}^{n-d} \left(b_{0,j}(Y_1)+b_{0,j}(Y_2)\right)h^dk^{n-d-j} + O(h^{d-1}),$$
whence  \begin{eqnarray*} \lambda_{\rm CW} (Z,X/B,L) &=& \left( \nu_0 ^{b_{0,j}(Y_1)+b_{0,j}(Y_2)} \otimes \left(\rho_{0,j}(Y_2) \otimes \rho_{0,j}(Y_2) \right)^{-a_0}\right)^{\frac{1}{a_0^2}} \\ &=& \left( \nu_0 ^{b_{0,j}(Y_1)} \otimes \rho_{0,j}(Y_a)^{-a_0}\right)^{\frac{1}{a_0^2}} \otimes \left( \nu_0 ^{b_{0,j}(Y_2)} \otimes \rho_{0,j}(Y_2)^{-a_0}\right)^{\frac{1}{a_0^2}} \\ &=& \lambda_{\rm CW} (Y_1,X/B,L) \otimes \lambda_{\rm CW} (Y_2,X/B,L).\end{eqnarray*}

Now let $\mathcal I_Z = \mathcal I_Y^m$ for some $m > 0$. By expansions \eqref{eq::rk_sub} and \eqref{eq::det_sub} we get $$ {\rm rank}\, \pi_! \left( L^h/\mathcal I_Z^k L^h \right) = \sum_{i=0}^d \sum_{j=0}^{n-d} b_{i,j}m^{n-d-j} h^{d-i} k^{n-d-j}$$ and $$\det \pi_!(L^{h}/\mathcal I_Z^k L^{h}) = \bigotimes_{i=0}^{d+1} \bigotimes_{j=0}^{n-d} \rho_{i,j}^{m^{n-d-j}h^{d+1-i}k^{n-d-j}}$$ respectively. Thus by definition of CW--line we get $$ \lambda_{\rm CW} (Z,X/B,L) = \left(\nu_0^{b_{0,1}m^{n-d-1}} \otimes \rho_{0,1}^{-m^{n-d-1}a_0} \right)^\frac{1}{a_0^2} = \lambda_{\rm CW}(Y,X/B,L)^{m^{n-d-1}}.$$

The general case follows easily by induction on $s$.
\end{proof}


\section{CM--line of blow-ups}\label{sec::CM_blup}

In this section we give an expression of the CM--line of a polarized flat family blown-up along a flat closed sub-scheme, in terms of the CM--line of the base family and the CW--line of the center of the blow-up.

Let $\pi: X \to B$ be a flat morphism of projective varieties with $B$ irreducible and relative dimension $\dim(X/B) = n$ and let $L$ be a relatively ample line bundle $X$. Moreover, let $i: Y \hookrightarrow X$ be the inclusion of a subscheme flat over $B$ (via $\pi$) with relative dimension $\dim(Y/B)=d<n-1$. 

Now let $\beta : \tilde X \to X$ be the blow-up of $X$ along $\mathcal I_Y$with exceptional (invertible) sheaf $\mathcal O_{\tilde X}(1)$. By hypothesis on $\pi$ and $Y$ we can conclude that $\tilde \pi = \pi \circ \beta$ is flat. We set $$ L_r = \beta^* L^r \otimes \mathcal O_{\tilde X}(1). $$ For $r$ sufficiently large $L_r$ is relatively ample. Moreover, for all $k\gg 0$, we have the identification $ \tilde \pi_! (L_r^k) = \pi_! (\mathcal I_Y^kL^{kr})$ of locally free sheaves on $B$. Thus, always for $k$ sufficiently large, we have the following exact sequence $$ 0 \to \mathcal I_Y^k L^{kr} \to L^{kr} \to L^{kr} / \mathcal I_Y^k L^{kr} \to 0,$$ whence \begin{equation}\label{eq::push_L_r^k}\tilde \pi_! (L_r^k) = \pi_!(L^{kr}) - \pi_! (L^{kr} /\mathcal I_Y^k L^{kr}). \end{equation}

\begin{thm}\label{thm::lCM_asym}
For $r$ sufficiently large we have the following asymptotic expansion 
$$ \lambda_{\rm CM}(\tilde X/B, L_r) = \lambda_{\rm CM}(X/B,L) \otimes \lambda_{\rm CW}(Y,X/B,L)^{\frac{1}{r^{n-d-1}}} \otimes O\left(\frac{1}{r^{n-d}}\right).$$
\end{thm}

\begin{proof}
By equation \eqref{eq::push_L_r^k} and expansions \eqref{eq::det_av} and \eqref{eq::det_sub} we get 
\begin{eqnarray*} 
\det \tilde \pi_! (L_r^k) &=& \det \pi_! (L^{kr}) \otimes \left( \det \pi_! (L^{kr} / \mathcal I_Y^k L^{kr})\right)^{-1} \\ 
&=& \bigotimes_{i=0}^{n+1} \nu_i^{(kr)^{n+1-i}} \otimes \bigotimes_{i=0}^{d+1} \bigotimes_{j=0}^{n-d} \rho_{i,j}^{-(kr)^{d+1-i}k^{n-d-j}}  \\
&=& \bigotimes_{i=0}^{n+1} \nu_i^{(kr)^{n+1-i}} \otimes \bigotimes_{i=0}^{d+1} \bigotimes_{j=0}^{n-d} \rho_{i,j}^{-r^{d+1-i}k^{n+1-(i+j)}}  \\ 
&=& \left( \nu_0^{r^{n+1}} \otimes \rho_{0,0}^{-r^{d+1}} \right)^{k^{n+1}} \otimes \left( \nu_1^{r^n} \otimes \rho_{0,1}^{-r^{d+1}} \otimes \rho_{1,0}^{-r^d} \right)^{k^n} \otimes O(k^{n-1}),
\end{eqnarray*}
and analogously by expansions \eqref{eq::rk_av} and \eqref{eq::rk_sub} we have 
\begin{eqnarray*} 
{\rm rank}\, \tilde \pi_! (L_r^k) &=& {\rm rank}\, \pi_! (L^{kr}) - {\rm rank}\, \pi_! (L^{kr} / \mathcal I_Y^k L^{kr}) \\
&=& \sum_{j=0}^n a_j (kr)^{n-j} - \sum_{i=0}^d \sum_{j=0}^{n-d} b_{i,j} (kr)^{d-i}k^{n-d-j} \\
&=& \sum_{j=0}^n a_j (kr)^{n-j} - \sum_{i=0}^d \sum_{j=0}^{n-d} b_{i,j} r^{d-i}k^{n-(i+j)} \\
&=& (a_0r^n - b_{0,0}r^d)k^n + (a_1 r^{n-1} - b_{0,1}r^d - b_{1,0}r^{d-1} )k^{n-1} + O(k^{n-2}). 
\end{eqnarray*}
Thus by definition \ref{def::CMline} of CM--line we have 
\begin{eqnarray}\label{eq::lCM^a_0^2} 
\lambda_{\rm CM}(\tilde X/B,L_r)^{(a_0r^n - b_{0,0}r^d)^2} &=& \left( \nu_0^{r^{n+1}} \otimes \rho_{0,0}^{-r^{d+1}} \right)^{-a_1 r^{n-1} + b_{0,1}r^d + b_{1,0}r^{d-1}} \\ 
\nonumber && \otimes \left( \nu_1^{r^n} \otimes \rho_{0,1}^{-r^{d+1}} \otimes \rho_{1,0}^{-r^d} \right)^{a_0r^n - b_{0,0}r^d}\\
\nonumber &=& \left(\nu_0^{-a_1} \otimes \nu_1^{a_0}\right)^{r^{2n}} \otimes \left(\nu_0^{b_{0,1}} \otimes \rho_{0,1}^{-a_0} \right)^{r^{n+d+1}} \otimes O\left( r^{n+d} \right), 
\end{eqnarray} 
whence
$$ \lambda_{\rm CM}(\tilde X/B,L_r) = \left(\nu_0^{-a_1} \otimes \nu_1^{a_0}\right)^{\frac{1}{a_0^2}} \otimes \left(\nu_0^{b_{0,1}} \otimes \rho_{0,1}^{-a_0} \right)^{\frac{1}{a_0^2 r^{n-d-1}}} \otimes O\left( \frac{1}{r^{n-d}} \right)$$
and the thesis follows by definitions of CM--line and CW--line.
\end{proof}

\begin{rem}
Although Theorem \ref{thm::lCM_asym} gives only an asymptotic expansion of the line bundle $\lambda_{\rm CM}(\tilde X/B, L_r)$, is clear that from equation \eqref{eq::lCM^a_0^2} is not hard to get an exact polynomial expansion of $\lambda_{\rm CM}(\tilde X/B, L_r)^{(a_0r^n-b_{0,0}r^d)^2}$ in terms of data on $X$ and $Y$. We do not state here such a more complete result due to the lack (to the best of the author's knowledge) of geometric meaning of some other line bundles appearing in the formula. Only about the lower order term, we notice that it is the limit  as $k \to \infty$ of the CM--line $\lambda_{\rm CM}(Y_k,L|_{Y_k})^{\frac{1}{(b_{0,0}k^n)^2}}$ associated to the polarized sub-scheme $Y_k \subset X$ cut out by $\mathcal I_Y^k$. These aspects will be objects of further studies by the author.  
\end{rem}


\section{Proofs of theorems \ref{thm::no_cscK_intro} and \ref{thm::no_eK_intro} }

In this section we consider a special case of situation of previous sections. As above we have a flat family $\pi:X \to B$ endowed with a relative ample line bundle $L$, but now we suppose $B=\mathbb C$ and we are given a $\mathbb C^\times$-action on $X$, covering the natural one on $\mathbb C$, and a linearization on $L$. If the fiber $M=\pi^{-1}(t)$ at $t \neq 0$ is non-singular, we recover the definition of test configuration for the polarized manifold $(M,L|_M)$ originally due to Donaldson:

\begin{defn}[\cite{Don02}]
A test configuration (of exponent one) for the polarized manifold $(M,A)$ is a $\mathbb C^\times$-equivariant flat family $\pi:X \to \mathbb C$ endowed with a relatively ample $\mathbb C^\times$-linearized line bundle $L$ such that $(\pi^{-1}(t),L|_{\pi^{-1}(t)}) \simeq (M,A)$ for all $t \neq 0$. The Futaki invariant of the central fiber $(\pi^{-1}(0),L|_{\pi^{-1}(0)})$ is called the Futaki invariant of the test configuration and will be denoted (with a little abuse of notation) by $F(X,L)$, leaving the $\mathbb C^\times$--action understood.
\end{defn}

Now consider a $n$-dimensional projective manifold $M$ and fix $N_1,\dots,N_s \subset M$ pairwise disjoint submanifolds of dimension $d<n-1$. Let $$\beta: \tilde M = {\rm Bl}_{N_1\cup \dots\cup N_s}(M) \to M$$ be the blowup of $M$ along the union $N_1 \cup \dots \cup N_s$. Clearly $\tilde M$ is smooth, moreover, denoting by $E_j$ the exceptional divisor over $N_j$, and fixed an ample line bundle $A$ on $M$, we get a collection of polarized manifolds \begin{equation}\label{eq::pol_blup} \left( \tilde M, \beta^*A^r\otimes(-\sum_{j=1}^sm_jE_j) \right),\end{equation} for $r\gg 0$.

Fix a $\mathbb C^\times$-action  $\alpha: \mathbb C^\times \to {\rm Aut}(M)$ on $M$ and a linearization on $A$. Since each $N_j$ move under the action of $\alpha$, is natural to expect that these data give a test configuration for each of the polarized manifolds \eqref{eq::pol_blup}. To construct this test configuration we proceed as follows. Set $X=M \times \mathbb C$ and let $\pi:X \to \mathbb C$ be the projection on the second factor and $L={\rm pr}_1^*A$ be the pull back of $A$ via the projection of $X$ on $M$. Clearly $\pi$ is flat and $L$ is relatively ample. Now consider the (closure of) the trajectory of each $N_j$ under the action of $\alpha$: $$ Y_j = \overline{ \left\{ (p,t) \in M\times \mathbb C^\times \,|\, \alpha(t^{-1})(p) \in N_j \right\}} \subset X, $$ denote by $\mathcal I_{Y_j} \subset \mathcal O_X$ its ideal sheaf and let $Y \subset X$ be the sub-scheme cut out by the ideal sheaf $\mathcal I_Y = \mathcal I_{Y_1}^{m_1} \cap \dots \cap \mathcal I_{Y_s}^{m_s}$. The restriction of $\pi$ to $Y$ is surjective and flat. For $t\neq 0$ the fiber $\pi|_Y^{-1}(t)$ is nothing but the union of subschemes $(N_j,\mathcal I_{N_j}^{m_j})$ moved by $\alpha(t)$, whereas $\pi|_Y^{-1}(0)$ is the flat limit of such union of sub-schemes as $t \to 0$.

Finally consider the blow-up $\beta : \tilde X \to X$ of $X$ along $\mathcal I_Y$. Since the latter is $\mathbb C^\times$-invariant, $\tilde X$ has an induced $\mathbb C^\times$-action. Moreover we are in the situation of section \ref{sec::CM_blup}. The map $\tilde \pi = \beta \circ \pi$ is flat and equivariant and the line bundle $$ L_r = \beta^*L^r \otimes \mathcal O_{\tilde X}(1) $$ has an induced linearization. Since each slice $M\times \{t\}$ is transversal to ${\rm supp}\, \mathcal I_Y$ for $t \neq 0$, the fiber $\tilde \pi^{-1}(t)$ of $\tilde \pi$ at $t$ is the blow-up of $M$ at the transformed submanifold $\alpha(t) \left( N_1 \cup \dots \cup N_s \right)$, and  we have the isomorphism $$ \left( \tilde \pi^{-1} (t) , \left( \beta^*L^r \otimes (-E) \right)|_{\tilde \pi^{-1}(t)} \right) \simeq \left( \tilde M, \beta^*A^r\otimes(-\sum_{j=1}^sm_jE_j) \right).$$ Thus the family $\tilde \pi: \tilde X \to \mathbb C $ polarized with $L_r$ is a test configuration for the blown-up polarized manifold \eqref{eq::pol_blup} with $r \gg 0$. At this point we notice that in general $\tilde \pi^{-1}(0)$ is not the blow-up of $M$ at the limit sub-scheme $$ \lim_{t \to 0} \alpha(t) \left( N_1 \cup \dots \cup N_s \right) = \pi|_Y^{-1}(0). $$ This phenomenon is well discussed in \cite{Sto08} in case where $\dim(N_j)=0$. 

Now we apply results of sections \ref{sec::CM&CW} and \ref{sec::CM_blup} to get an asymptotic expansion of the Futaki invariant of the test configuration $\left(\tilde X / \mathbb C, L_r\right)$ we just constructed. In particular by theorem \ref{thm::lCM_asym} we get the following

\begin{cor}\label{cor::asym_F(r)} 
Let $F(\tilde X,L_r), F(M,A)$ be the Futaki invariants of the test configurations constructed above respectively for the manifolds $\left( \tilde M,\beta^* A^r \otimes \left( -\sum_{j=1}^s m_j E_j \right) \right)$ and $(M,A)$. For $r \gg 0$ we have $$ F(\tilde X,L_r) = F(M,A) + \frac{1}{r^{n-d-1}}\sum_{j=1}^s w_{\rm CW}(N_j,M,A)m_j^{n-d-1} + O\left(\frac{1}{r^{n-d}}\right).$$
\end{cor} 

\begin{proof}
Since $\mathcal I_Z = \mathcal I_{Y_1}^{m_1} \cap \dots \cap \mathcal I_{Y_s}^{m_s}$ by theorem \ref{thm::lCM_asym} and lemma \ref{lem::lAC_mor} for $r \gg 0$ we get $$ \lambda_{\rm CM}(\tilde X/\mathbb C,L_r) = \lambda_{\rm CM}(X/\mathbb C,L) \otimes \bigotimes_{j=1}^s \lambda_{\rm CW}(Y_j,X/\mathbb C,L)^{\left( \frac{m_j}{r} \right)^{n-d-1}} \otimes O\left(\frac{1}{r^{n-d}}\right).$$ Thus the statement follows by theorem \ref{thm::PT_Fut} taking the weight of the $\mathbb C^\times$-action on the central fiber $\left( \lambda_{\rm CM}(\tilde X/\mathbb C,L_r) \right)_0$. 
\end{proof}

\begin{rem}
More generally we can start with an arbitrary test configuration $(X/\mathbb C,L)$ for $(M,A)$ with $\mathbb C^{\times}$--action $\alpha: \mathbb C^{\times} \to {\rm Aut}(X)$ and projection $\pi : X \to \mathbb C$. Embedded $N_j \subset \pi^{-1}(1)$ in the fiber over $t=1$, we set $$ Y_j = \overline{\left\{ p \in X \,|\, \alpha(\pi(p)^{-1})(p) \in N_j  \right\}} \subset X.$$ Repeating the argument above we get the proofs of corollaries \ref{cor::fut_exp_intro} and \ref{cor::fut_exp_mult_intro}.     
\end{rem}

Finally we are in position to give the following

\begin{proof}[proof of theorem \ref{thm::no_cscK_intro}]
Since $N$ is Chow--unstable, there is a one--parameter subgroup $\alpha: \mathbb C^\times \to {\rm Aut}(M)$ that linearizer on $A$ such that the Chow--weight $w_{\rm CW}(N,M,A)$ discussed in remark \ref{rem::CW_subvar} is positive. Following the construction illustrated above, with this one--parameter subgroup we can construct a test configuration $(\tilde X/\mathbb C,L_r)$ for $$ \left( \tilde M, \beta^*A^r\otimes(-\sum_{j=1}^sm_jE_j) \right)$$ with $r\gg 0$. Moreover by hypothesis $(M,A)$ admits a cscK metric, thus $F(M,A)=0$ vanish for all $\mathbb C^\times$--actions and by corollary \ref{cor::asym_F(r)} for $r \gg 0$ we get $$ F(\tilde X, L_r) = \frac{w_{\rm CW}(N,A)}{r^{n-d-1}} + O\left(\frac{1}{r^{n-d}}\right),$$ and the statement follows by theorem \ref{thm::cscK_imp_K-stab_intro}.   
\end{proof}

\begin{proof}[proof of theorem \ref{thm::no_eK_intro}]
The situation is similar to above. Since $N$ is relatively Chow--unstable, we can find a destabilizing one--parameter subgroup $\alpha : \mathbb C^\times \to Z_{{\rm Aut}(M)}(\tilde T)/\tilde T$, where $\tilde T \subset {\rm Aut}(M)_N$ is a fixed maximal torus of the stabilizer of $N$ in $M$ and $Z_{{\rm Aut}(M)}(\tilde T)$ is the identity component of the centralizer of $\tilde T$ in ${\rm Aut}(M)$. To fix an embedding $Z_{{\rm Aut}(M)}(\tilde T)/\tilde T \hookrightarrow {\rm Aut}(M)$, we proced as follows. Since ${\rm Lie}(Z_{{\rm Aut}(M)}(\tilde T))$ is a subalgebra of holomorphic vector field on $M$, we can consider the restriction to ${\rm Lie}(Z_{{\rm Aut}(M)}(\tilde T))$ of the Futaki-Mabuchi scalar product (see \cite{FutMab95}) of the class $c_1(A)$ and we define $G_{\tilde T}$ to be the subgroup generated by ${\rm Lie}(\tilde T)^\perp$. Thus we can consider $\alpha$ a one--parameter subgroup of ${\rm Aut}(M)$ commuting with $\tilde T$ and with generating vector field $\dot \alpha \in {\rm Lie}(\tilde T)^\perp$.

Moreover, by remark \ref{rem::CW_subvar} the Chow-weights of $N$ and its flat limit $N_0=\lim_{t \to 0}\alpha(t)\cdot N$ define holomorphic vector fields $\gamma$ and $\gamma_0$ on $M$ via the aforementioned scalar product. Thus by instability hypothesis and definition of $G_{\tilde T}$ we have \begin{equation}\label{eq::ortho_alpha_gamma} w_{\rm CH}(N,M,A) = \langle \dot\alpha, \gamma_0\rangle_{c_1(A)} >0 \qquad \mbox{and} \qquad \langle \dot\alpha, \gamma\rangle_{c_1(A)} =0. \qquad \end{equation}

Thanks to the invariance hypotesis on $N$ we have a privilegiate choice for the extremal vector field $\eta$ on $M$ of the class $c_1(A)$. With this choice, integrating the extremal field $r^{-2}\eta$ of the class $c_1(A^r)$, for each integer $r>0$ we get an extremal action $\chi^{(r)}: \mathbb C^\times \to \tilde T$ (see \cite{Sze06,FutMab95} for definition) on the manifold $M$. On the other hand, via the identification ${\rm Aut}(M)_N = {\rm Aut}(\tilde M)$ we can regard also the extremal action $\tilde \chi^{(r)} :\mathbb C^\times \to \tilde T$ of the class $c_1\left(A_r\right)$ on the blown up manifold $\tilde M$ as a $\mathbb C^\times$-action on $M$ (perhaps replacing $A_r=\beta^*A^r\otimes(-\sum_{j=1}^sm_jE_j)$ with some tensor power). 
The actions $\chi^{(r)}$ and $\tilde \chi^{(r)}$ induce $\mathbb C^\times$--actions on test configurations $X$ and $\tilde X$ constructed as above. 

Following Sz\'ekelyhidi \cite{Sze06} we have to compute the {\it corrected} Futaki invariant \begin{equation}\label{eq::fut_cor} F(\tilde X, L_r) - \langle \tilde \alpha_0, \tilde \chi^{(r)}_0 \rangle_{c_1(L_r|_{\tilde X_0})}\end{equation} where $\tilde \alpha_0$ and $\tilde \chi_0^{(r)}$ are the induced actions on the central fiber of $\tilde X$ respectively by the actions $\alpha$ and $\tilde \chi^{(r)}$ on $M$, and we denote the generalization to possibly singular and non--reduced schemes of the Futaki-Mabuchi product with the same notation. The key observation is that \begin{equation}\label{eq::asym_bil} \langle \tilde \alpha_0, \tilde \chi^{(r)}_0 \rangle_{c_1(L_r|_{\tilde X_0})} = \langle \alpha, \chi \rangle_{c_1(A)} + O\left(\frac{1}{r^{n-d}}\right)\end{equation} as $r \gg 0$, where $\langle \alpha, \chi \rangle_{c_1(A)} = \langle \dot \alpha , \eta \rangle = F(M,A)$ since $X$ is a product configuration. Assuming \eqref{eq::asym_bil} for the moment, together with the corollary \ref{cor::asym_F(r)} we get $$ F(\tilde X, L_r) - \langle \tilde \alpha_0, \tilde \chi^{(r)}_0 \rangle_{c_1(L_r|_{\tilde X_0})} =  \frac{w_{\rm CW}(N,A)}{r^{n-d-1}} + O\left(\frac{1}{r^{n-d}}\right),$$ and the statement follows by \cite[theorem 3.3.2]{Sze06}.

To prove \eqref{eq::asym_bil} consider the embedding $\iota_{L_r} : \tilde X \hookrightarrow \mathbb P^{N_r} \times \mathbb C$. Thanks to linearizations on $L_r$, $\tilde T$ and $\alpha$ act on $\mathbb P^{N_r} \times \mathbb C$ (in particular $\alpha$ acts via the natural action of $\mathbb C^\times$ on $\mathbb C$ and $\tilde T$ via the trivial action) and $\iota$ is both $\tilde T$ and $\alpha$--equivariant. Now let $\omega_{E}$ be the standard K\"ahler metric on $\mathbb C$ and let $\omega_{FS(N_r)}$ be a Fubini-Study metric on $\mathbb P^{N_r}$ such that $\alpha|_{S^1}$ acts by isometries. Let $\phi_\alpha^{(r)}$ be a K\"ahler potential for $\alpha$ with respect to $\omega_{FS(N_r)}$. 

For each integer $k>0$ we have \begin{eqnarray}\label{eq::lim_central} \int_{\tilde X_0} \left(\phi_\alpha^{(r)}\right)^k \frac{\omega_{FS(N_r)}^n}{n!} 
&=& \lim_{t \to 0} \int_{\tilde X_t} \left( \phi_\alpha^{(r)} + |z|^2 \right)^k \frac{\left(\omega_{FS(N_r)} + \omega_E\right)^n}{n!} \nonumber \\
&=& \lim_{t \to 0} \int_{\tilde X_t} \left( \phi_\alpha^{(r)} + |t|^2 \right)^k \frac{\omega_{FS(N_r)}^n}{n!} \nonumber \\
&=& \lim_{t \to 0} \int_{\tilde X_t} \left( \phi_\alpha^{(r)} \right)^k \frac{\omega_{FS(N_r)}^n}{n!}. \end{eqnarray}
Now, let $\omega$ be a K\"ahler metric on $M$ in the class $c_1(A)$ and let $\psi_\alpha$ be a K\"ahler potential for $\alpha$ with respect to $\omega$. For each $t \neq 0$ have \begin{equation}\label{eq::scal_asympt}\int_{\tilde X_t} \left( \phi_\alpha^{(r)} \right)^2 \frac{\omega_{FS(N_r)}^n}{n!} - \left( \int_{\tilde X_t} \phi_\alpha^{(r)} \frac{\omega_{FS(N_r)}^n}{n!} \right)^2 = r^{n+2} \left( \int_M \psi_\alpha^2 \frac{\omega^n}{n!} - \left( \int_M \psi_\alpha \frac{\omega^n}{n!} \right)^2 \right) + O_t(r^{d+2})\end{equation} 
Combinig \eqref{eq::lim_central} with \eqref{eq::scal_asympt} by \cite[Proposition 3]{Don05} and definition of the generalized Futaki-Mabuchi scalar product we get \begin{equation}\label{eq::fut_mab_asym} \langle \lambda_1, \lambda_2 \rangle_{c_1(L_r|_{X_0})} = r^2 \langle \dot \lambda_1, \dot \lambda_2 \rangle_{c_1(A)} + O\left( \frac{1}{r^{n-d-2}}\right),\end{equation} for each pair $\lambda_1, \lambda_2$ of one--parameter subgroups of ${\rm Aut}(M)$ generated by the vector fields $\dot \lambda_1, \dot \lambda_2 \in {\rm Lie}(\tilde T) + \mathbb C \dot \alpha$.
Since the extremal vector field (in the fixed torus) is by definition the dual of the Futaki invariant with respect to the Futaki-Mabuchi scalar product, denoting by $\eta_r$ the extremal field of the class $c_1(A_r)$, by \eqref{eq::fut_mab_asym} and corollary \ref{cor::asym_F(r)} for $r\gg 0$ we get $$ \eta_r = \frac{1}{r^2}\eta + \frac{1}{r^{n-d+1}}\gamma + O\left( \frac{1}{r^{n-d+2}} \right),$$ whence $$ \langle \tilde \alpha_0, \tilde \chi^{(r)}_0 \rangle_{c_1(L_r|_{\tilde X_0})} = \langle \dot \alpha, \eta \rangle_{c_1(A)} + \frac{1}{r^{n-d-1}} \langle \dot \alpha, \gamma \rangle_{c_1(A)} + O\left( \frac{1}{r^{n-d}} \right) $$ and \eqref{eq::asym_bil} follows thanks to \eqref{eq::ortho_alpha_gamma}.
\end{proof}



\begin{thebibliography}{99}

\bibitem{ApoCalGauTon05} V. Apostolov, D. M. J. Calderbank, P. Gauduchon and C. W. T{\o}nnesen-Friedman,
                         \emph{Hamiltonian 2-forms in K\"ahler geometry, III Extremal metrics and stability}.
                         Invent. Math. {\bf 173} (2008), no. 3, 547--601.
                         arXiv:math/0511118v1 [math.DG].
                         
\bibitem{ArePacSin07} C. Arezzo, F. Pacard and M. Singer,
                      \emph{Extremal metrics on blow ups}.
                      arXiv:math/0701028v1 [math.DG].

\bibitem{Aub76} T. Aubin,
                \emph{Equations du type Monge--Ampère sur le varietes k\"ahleriennes compactes}.
                C. R. Acad. Sci. Paris {\bf 283} (1976) 119--121.

\bibitem{BurDeB88} D. Burns and P. De Bartolomeis, 
                   \emph{Stability of vector bundles and extremal metrics}. 
                   Invent. Math. {\bf 92} (1988), no. 2, 403--407.

\bibitem{Cal82} E. Calabi,
                \emph{Extremal K\"ahler metrics}. 
                Seminar on Differential Geometry, Ann. of Math. Stud. {\bf 102} (1982) 259--290. 
                Princeton University Press, Princeton (NJ)

\bibitem{Cal85} E. Calabi, 
                \emph{Extremal K\"ahler metrics. II}.
                Differential geometry and complex analysis, 95--114.
                Springer, Berlin, 1985.

\bibitem{Car60} P. Cartier
                \emph{Sur un th\'eor\`eme de Snapper}. 
                Bull. Soc. Math. France {\bf 88} (1960) 333--343.

\bibitem{CheLeBWeb07} X. X. Chen, C. LeBrun and B. Weber,
                      \emph{On Conformally K\"ahler, Einstein Manifolds}.
                      J. Amer. Math. Soc. {\bf 21} (2008), 1137-1168.

\bibitem{Dol03} I. Dolgachev,
                \emph{Lectures on invariant theory}.
                London Mathematical Society Lecture Note Series, 296.
                Cambridge University Press, 2003.

\bibitem{Don02} S. K. Donaldson,
                \emph{Scalar curvature and stability of toric varieties}.
                J. Differential Geom. {\bf 59} (2002), no. 2, 289--349.

\bibitem{Don05} S. K. Donaldson,
                \emph{Lower bounds on the Calabi functional}.
                J. Differential Geom. {\bf 70} (2005), no. 3, 453--472.

\bibitem{FinRos06} J. Fine and J. Ross,
                   \emph{A note on positivity of the CM line bundle}.  
                   Int. Math. Res. Not. (2006), Art. ID 95875, 14 pp.

\bibitem{Fut83} A. Futaki,
                \emph{An obstruction to the existence of Einstein-K\"ahler metrics}.
                Invent. Math. {\bf 73} (1983), no. 3, 437--443.

\bibitem{FutMab95} A. Futaki and T. Mabuchi,
                   \emph{Bilinear forms and extremal {K}\"ahler vector fields associated with {K}\"ahler classes}.
                   Math. Ann. 301 (1995), n.2, 199--210

\bibitem{He07} W. He,
               \emph{Remarks on the existence of bilaterally symmetric extremal K\"ahler metrics on $\mathbb{CP}^2 \# 2\overline{\mathbb{CP}}^2$}.
                Int. Math. Res. Not. IMRN 2007, no. 24, Art. ID rnm127, 13 pp.
                
\bibitem{Kir84} F. C. Kirwan,
                \emph{Cohomology of quotients in symplectic and algebraic geometry}.
                Mathematical Notes, 31.\
                Princeton University Press, 1984.

\bibitem{KnuMum76} F. F. Knudsen and D. Mumford,
                   \emph{The projectivity of the moduli space of stable curves I: Preliminaries on ``det'' and ``Div''}.
                   Math. Scand. {\bf 39} (1976), no. 1.

\bibitem{LeBSim94} C. LeBrun and S. R. Simanca,
                   \emph{Extremal K\"ahler metrics and complex deformation theory}.
                   Geom. Funct. Anal. {\bf 4} (1994), no. 3, 298--336.

\bibitem{Lev85} M. Levine,
                \emph{A remark on extremal K\"ahler metrics}.
                J. Differential Geom. {\bf 21} (1985), no. 1, 73--77.
                
\bibitem{Mab04} T. Mabuchi,
                \emph{Stability of extremal K\"ahler manifolds}.
                Osaka J. Math. {\bf 41} (2004), no. 3, 563--582.

\bibitem{Mat57} Y. Matsushima,
                \emph{Sur la structure du groupe d'homéomorphismes analytiques d'une certaine varieté k\"ahlerienne}.
                Nagoya Math. J. {\bf 11} (1957) 145--150.

\bibitem{Muk03} S. Mukai,
                \emph{An introduction to invariants and moduli}.
                Cambridge Studies in Advanced Mathematics, 81.
                Cambridge University Press, 2003

\bibitem{Mum77} D. Mumford,
                \emph{Stability of projective varieties}.
                L'enseignement matematique {\bf XXIII} (1977), no. 1--2, 39--110.

\bibitem{MumFogKir94} D. Mumford, J. Fogarty and F. Kirwan,
                      \emph{Geometric invariant theory}.
                      Third edition. Ergebnisse der Mathematik und ihrer Grenzgebiete, {\bf 34} (1994). 
                      Springer-Verlag, Berlin.

\bibitem{Nad90} A. M. Nadel,
                \emph{Multiplier ideal sheaves and K\"ahler-Einstein metrics of positive scalar curvature}.
                Ann. of Math. (2) {\bf 132} (1990), no. 3, 549--596.

\bibitem{PauTia06} S. T. Paul and G. Tian,
                   \emph{CM stability and the generalized Futaki invariant I}.
                   arXiv:math/0605278v5 [math.AG].

\bibitem{RosTho07} J. Ross and R. Thomas
                   \emph{A study of the Hilbert-Mumford criterion for the stability of projective varieties}.
                   J. Algebraic Geom. {\bf 16} (2007), no. 2, 201--255.

\bibitem{Sto08} J. Stoppa,
                \emph{Unstable blowups}.
                arXiv:math/0702154v2 [math.AG]. To appear in J. Algebraic Geom.

\bibitem{Sto08bis} J. Stoppa,
                   \emph{K-stability of constant scalar curvature K\"ahler manifolds}.
                   arXiv:0803.4095v1 [math.AG].

\bibitem{Sze06} G. Sz\'ekelyhidi,
                \emph{Extremal metrics and K-stability}.
                Imperial College, University of London,
                Ph.d. thesis (2006) and arXiv:math/0611002v1 [math.DG].

\bibitem{Sze08} G. Sz\'ekelyhidi,
                \emph{The Kahler-Ricci flow and K-stability}.
                arXiv:0803.1613v1 [math.DG].

\bibitem{Tho06} R. P. Thomas, 
                \emph{Notes on GIT and symplectic reduction for bundles and varieties}.
                JDG Conference on Geometry and Topology. Yau, S.-T., editor, 
                Surveys in Differential Geometry X, International Press (2006).

\bibitem{Tia97} G. Tian,
                \emph{K\"ahler-Einstein metrics with positive scalar curvature}.
                Invent. Math. {\bf 130} (1997), no. 1, 1--37.

\bibitem{TiaYau87} G. Tian and S.-T. Yau,
                   \emph{K\"ahler-Einstein metrics on complex surfaces with $C\sb 1>0$}.  
                   Comm. Math. Phys. {\bf 112} (1987), no. 1, 175--203.

\bibitem{Ton98} C. W. T{\o}nnesen-Friedman, 
                \emph{Extremal K\"ahler metrics on minimal ruled surfaces}.
                J. Reine Angew. Math. {\bf 502} (1998), 175--197.

\bibitem{Wan04} X. Wang,
                \emph{Moment map, Futaki invariant and stability of projective manifolds}.  
                Comm. Anal. Geom. {\bf 12} (2004), no. 5, 1009--1037.

\bibitem{Yau78} S.-T. Yau
                \emph{On the Ricci curvature of a compact K\"ahler manifold and the complex Monge--Ampère equation I}.
                Comm. Pure Appl. Math. {\bf 31} (1978), no. 3, 339--441.

\end{thebibliography}
\end{document}